\documentclass[reqno]{amsart}
\usepackage{amsmath,amsthm,amssymb,amscd,float}
\usepackage{subcaption}
\usepackage[shortlabels]{enumitem}
\usepackage{youngtab}
\usepackage{tikz}
\usetikzlibrary{shapes.geometric, arrows.meta, positioning, fit}
\usepackage{xcolor}
\usetikzlibrary{arrows,matrix}
\usepackage{amsfonts,hyperref} 
\usepackage{tabularx}
\newtheorem{theorem}{Theorem}[section]

\usepackage[english]{babel}
\usepackage{dutchcal}
\allowdisplaybreaks

\usepackage{ulem}

\title[Overpartitions with repeated smallest non-overlined part]{Combinatorial proofs of some identities on overpartitions with repeated smallest non-overlined part}

\author[N. D. Baruah]{Nayandeep Deka Baruah}
\address{Department of Mathematical Sciences, Tezpur University, Assam, India, PIN-784028}
\email{nayan@tezu.ernet.in, nayandeeptezu@gmail.com}

\author[H. Li]{Haijun Li}
\address{College of Mathematics and Statistics, Chongqing University, Chongqing 401331, P.R. China}
\email{lihaijun@cqu.edu.cn, lihaijune@163.com}

\author[P. J. Mahanta]{Pankaj Jyoti Mahanta}
\address{Department of Mathematical Sciences, Tezpur University, Assam, India, PIN-784028}
\email{pjm2099@gmail.com, msp25007@tezu.ac.in}

\keywords{Integer partition, Overpartition, Combinatorial equality, bijective combinatorics}

\subjclass[2020]{05A17, 05A15, 11P81.}

\begin{document}

\begin{abstract}
Let $\overline{\mathrm{spt}}k(n)$ denote the number of overpartitions of $n$ where the smallest non-overlined part, say $s(\pi)$, appears $k$ times and every overlined part is bigger than $s(\pi)$. Let $\overline{\mathrm{spt}}k_o(n)$ denote the number of overpartitions of $n$ where the smallest non-overlined part appears $k$ times, every overlined part is bigger than $s(\pi)$ and all parts other than $s(\pi)$ are incongruent modulo $2$ with $s(\pi)$. Also, let $b_e(k,n)$ (resp., $b_o(k,n)$) denote the number of overpartitions of $n$ counted by $\overline{\mathrm{spt}}k_o(n)$ where the number of parts greater than $s(\pi)$ is even (resp., odd), and let
$$\overline{\mathrm{spt}}k_o'(n)=b_e(k,n)-b_o(k,n).$$
Recently, Malik and Sarma (arXiv:2601.15601v1) expressed the generating functions of these partition functions in terms of linear combinations of $q$-series with polynomials in $q$ as coefficients. As corollaries, they derived some partition identities involving the functions for $k=1$ and sought for combinatorial proofs of their results. In this paper, we present some desired proofs.
\end{abstract}

\maketitle

\section{Introduction}
A {\it partition} $\pi$ of a positive integer $n$ is defined as a non-increasing sequence of positive integers $(\pi_1, \pi_2, \ldots, \pi_{\ell})$ such that $\pi_1+\pi_2+\cdots +\pi_{\ell}=n$. The terms $\pi_i$ are called the {\it parts} of $\pi$. For example, partitions of 4 are
$$(4), (3,1), (2,2), (2,1,1), (1,1,1,1).$$
For more details on integer partitions, see \cite{Andrews_1998}. An {\it overpartition} of $n$ is a partition of $n$ in which the first occurrence of each part may be overlined \cite{overpartitions}. For example, overpartitions of 4 are
\begin{align*}
& (4), (\overline{4}), (3,1), (\overline{3},1), (3,\overline{1}), (\overline{3},\overline{1}), (2,2), (\overline{2},2),\\
& (2,1,1), (\overline{2},1,1), (2,\overline{1},1), (\overline{2},\overline{1},1), (1,1,1,1), (\overline{1},1,1,1).
\end{align*}

In 2025, Andrews and El Bachraoui \cite{andrews2025generating} studied the number of partitions whose smallest part is repeated exactly $k$ times and the remaining parts are not repeated. They showed that the generating functions of these partition functions can be represented as linear combinations of the $q$-series with polynomials in $q$ as coefficients. As a particular case, they analytically derived several new identities and inequalities, and invited combinatorial proofs. In \cite{ndb2026pjm}, the first and third author gave combinatorial proofs for some of their identities and inequalities.

Recently, Malik and Sarma \cite{rishabh} extended the work of Andrews and El Bachraoui \cite{andrews2025generating} to the case of overpartitions. By proving their corresponding generating functions, they analytically established three partition identities. 

Before proceeding further, let us go through the following definitions.
\begin{itemize}
\item $\overline{\mathrm{spt}}k(n)$ denotes the number of overpartitions of $n$ where the smallest non-overlined part, say $s(\pi)$, appears $k$ times and every overlined part is bigger than $s(\pi)$.
\item $\overline{\mathrm{spt}}k_o(n)$ denotes the number of overpartitions of $n$ where the smallest non-overlined part appears $k$ times, every overlined part is bigger than $s(\pi)$ and all parts other than $s(\pi)$ are incongruent modulo $2$ with $s(\pi)$.
\item $\overline{p}_e(n)$ denotes the number of overpartitions of $n$ into even parts.
\item $\overline{p}_{ex}(n)$ denotes the number of overpartitions of $n$ with no non-overlined $1$'s.
\item $\overline{p}_{oex}(n)$ denotes the number of overpartitions of $n$ into odd parts with no non-overlined $1$'s.
\end{itemize}
We denote the sets of overpartitions counted by the partition functions above by $\overline{\mathrm{Spt}}k(n)$, $\overline{\mathrm{Spt}}k_o(n)$, $\overline{P}_e(n)$, $\overline{P}_{ex}(n)$, $\overline{P}_{oex}(n)$, respectively. In addition, let us look at the following definitions.
\begin{itemize}
\item Let $b_e(k,n)$ (resp., $b_o(k,n)$) denote the number of overpartitions $\pi$ of $n$ counted by $\overline{\mathrm{spt}}k_o(n)$ where the number of parts greater than $s(\pi)$ is even (resp., odd). Let
$$\overline{\mathrm{spt}}k_o'(n)=b_e(k,n)-b_o(k,n).$$
\item Let $c_e(n)$ (resp., $c_o(n)$) be the set of overpartitions counted by $\overline{p}_{oex}(n)$ where the number of parts is even (resp., odd). Let
$$\overline{p}_{oex}'(n)=c_e(n)-c_o(n).$$
\end{itemize}
We use $B_e(k,n)$, $B_o(k,n)$, $C_e(n)$, and $C_o(n)$ to denote the sets of overpartitions counted by $b_e(k,n)$, $b_o(k,n)$, $c_e(n)$, and $c_o(n)$, respectively.

The partition identities analytically established by Malik and Sarma \cite{rishabh} are as follows.

\begin{theorem}\cite[Corollary 4.1]{rishabh}\label{cor4.1}
For $n>1$, we have 
$$\overline{\mathrm{spt}}1(n)+\overline{\mathrm{spt}}1(n-1)= \overline{p}_{ex}(n).$$
\end{theorem}

\begin{theorem}\cite[Corollary 4.2]{rishabh}\label{cor4.2}
For $n>2$, we have $$\overline{\mathrm{spt}}1_o(n)+\overline{\mathrm{spt}}1_o(n-2)=2\overline{p}_e(n-1)+\overline{p}_{oex}(n-1).$$
\end{theorem}

\begin{theorem}\cite[Corollary 4.3]{rishabh}\label{cor4.3}
For $n>2$, we have $$\overline{\mathrm{spt}}1_o'(n)+\overline{\mathrm{spt}}1_o'(n-2)=-\overline{p}_{oex}'(n-1).$$
\end{theorem}

They sought combinatorial proofs of the above results. In Section \ref{sec:proof123}, we provide such proofs. We present the following theorem as a refined version. In Section \ref{sec:proof4} we prove Theorem \ref{thm4} combinatorially and discuss how Theorems \ref{cor4.2} and \ref{cor4.3} can be derived from Theorem \ref{thm4}.

\begin{theorem}\label{thm4}
For $n>2$, we have
\begin{align}
b_e(1,n)+b_e(1,n-2) & =\overline{p}_e(n-1)+c_o(n-1),\label{id:e}\\
b_o(1,n)+b_o(1,n-2) & =\overline{p}_e(n-1)+c_e(n-1).\label{id:o}
\end{align}
\end{theorem}

\section{Combinatorial proof of Theorems \ref{cor4.1}, \ref{cor4.2} and \ref{cor4.3}}\label{sec:proof123}

\begin{proof}[Proof of Theorem \ref{cor4.1}]
Let $s_2(\pi)$ denote the part of a overpartition immediately greater than $s(\pi)$. We will retain this notation for the rest of the article.

Now, we construct the following maps from $\overline{\mathrm{Spt}}1(n)$ to $\overline{P}_{ex}(n)$ as:
$$\pi
\begin{cases}
\xrightarrow{f_1}	\pi,
 & \hspace{4mm} \text{ if } s(\pi)> 1,\\
\xrightarrow{f_2} \left(\ldots,\overline{s(\pi)}\right),
& \hspace{4mm} \text{ if } s(\pi)=1.
\end{cases}
$$

Next, we construct the following maps from $\overline{\mathrm{Spt}}1(n-1)$ to $\overline{P}_{ex}(n)$ as follows:
$$\pi
\begin{cases}
\xrightarrow{f_3}	\left(\ldots,\overline{s(\pi)+1}\right),
 & \hspace{4mm} \text{ if } s_2(\pi)-s(\pi)> 1,\\
\xrightarrow{f_4} \left(\ldots,s(\pi)+1\right),
& \hspace{4mm} \text{ if } s_2(\pi)-s(\pi)=1.
\end{cases}
$$

Note that,
\begin{itemize}
    \item The smallest part of $f_1(\pi)$ is greater than 1, non-overlined, and appears only once.
    \item The smallest part of $f_2(\pi)$ is $\overline{1}$.
    \item The smallest part of $f_3(\pi)$ is greater than 1, and overlined.
    \item The smallest part of $f_4(\pi)$ is greater than 1, non-overlined, and $s(\pi)$ appear at least twice or $\overline{s(\pi)}$ appears.
\end{itemize}
Hence, we obtain all overpartitions in $\overline{P}_{ex}(n)$, which completes the proof.

An example is shown in the following figure.

\noindent	\begin{tikzpicture}[
		box/.style={rounded corners, draw, minimum width=2.5cm, minimum height=5cm, thick},
		innerbox/.style={rounded corners, draw, minimum width=1.8cm, inner sep=5pt, align=center},
		arrow/.style={-Stealth, thick}]
		
		\node[box] (Spt6) {};
		\node[below=2pt of Spt6] {$\overline{\mathrm{Spt}}1(6)$};
		
		\node[innerbox, anchor=north] (Spt6_top) at ([yshift=-10pt]Spt6.north) {
			$(6)$ \\ $(4,2)$ \\ $(\overline{4},2)$
		};
		\node[innerbox, anchor=south] (Spt6_bottom) at ([yshift=7pt]Spt6.south) {
			$(5,1)$ \\ $(\overline{5},1)$ \\ $(3,2,1)$ \\ $(\overline{3},2,1)$ \\ $(3,\overline{2},1)$ \\ $(\overline{3},\overline{2},1)$ 
		};
		
		\node[box, right=0.5cm of Spt6, minimum height=4.5cm] (Spt5) {};
		\node[below=2pt of Spt5] {$\overline{\mathrm{Spt}}1(5)$};
		
		\node[innerbox, anchor=north] (Spt5_top) at ([yshift=-15pt]Spt5.north) {
			$(5)$ \\ $(4,1)$ \\ $(\overline{4},1)$
		};
		\node[innerbox, anchor=south] (Spt5_bottom) at ([yshift=10pt]Spt5.south) {
		$(3,2)$ \\ $(\overline{3},2)$ \\ $(2,2,1)$ \\ $(\overline{2},2,1)$
		};
		
		\node[box, right=1.5cm of Spt5, minimum width=5cm,minimum height=5.5cm] (Pex6) {};
		\node[below=9pt of Pex6] {$\overline{P}_{ex}(6)$};
		
		\node[innerbox] (Pex_TL) at ([xshift=-1.2cm, yshift=1.6cm]Pex6.center) {
$(6)$ \\ $(4,2)$ \\ $(\overline{4},2)$
		};
		\node[innerbox] (Pex_TR) at ([xshift=1.2cm, yshift=1.6cm]Pex6.center) {
		$(\overline{6})$ \\	$(4,\overline{2})$ \\ $(\overline{4},\overline{2})$
		};
		\node[innerbox] (Pex_BL) at ([xshift=-1.2cm, yshift=-0.9cm]Pex6.center) {
	$(5,\overline{1})$ \\ $(\overline{5},\overline{1})$ \\ $(3,2,\overline{1})$ \\ $(\overline{3},2,\overline{1})$ \\ $(3,\overline{2},\overline{1})$ \\ $(\overline{3},\overline{2},\overline{1})$ 
		};
		\node[innerbox] (Pex_BR) at ([xshift=1.2cm, yshift=-1cm]Pex6.center) {
		$(3,3)$ \\ $(\overline{3},3)$ \\ $(2,2,2)$ \\	$(\overline{2},2,2)$
		};
		
		
		\draw[arrow] (Spt6_top.north) to [out=40, in=145] (Pex_TL.north);
		
		\draw[arrow] (Spt5_top.north) to [out=40, in=120] (Pex_TR.north);
		
		\draw[arrow] (Spt6_bottom.south) to [out=-30, in=220] (Pex_BL.south);
		
		\draw[arrow] (Spt5_bottom.south) to [out=-30, in=230] (Pex_BR.south);
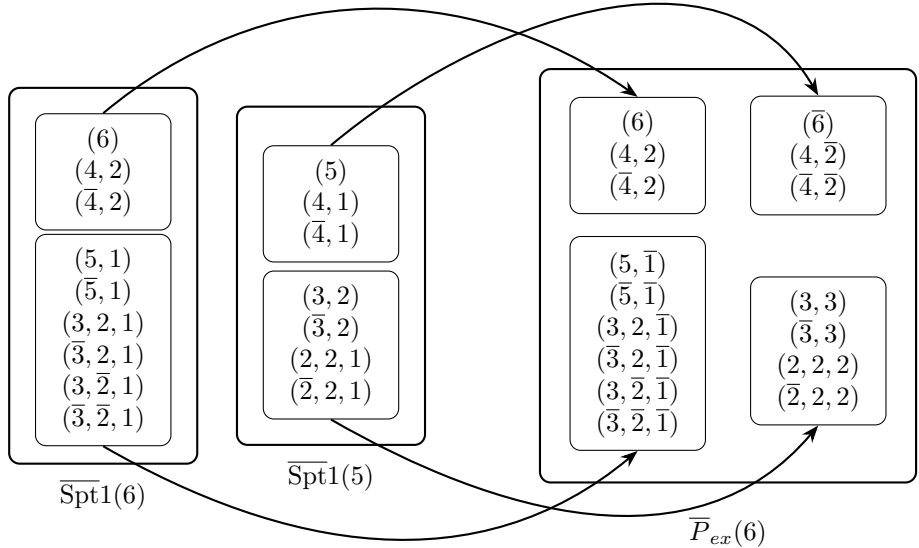
\node at (5, -4.1) {Example for $n= 6$.};
	\end{tikzpicture}

\end{proof}

\begin{proof}[Proof of Theorem \ref{cor4.2}]
For each $\pi\in \overline{\mathrm{Spt}}1_o(n)$ with $s(\pi)=1$, removing the part 1 from $\pi$ yields a distinct overpartition in $\overline{P}_e(n-1)$. The reverse process also yields a distinct overpartition in $\overline{\mathrm{Spt}}1_o(n)$ for each overpartition in $\overline{P}_e(n-1)$. Consequently, the total number of such overpartition equals $\overline{p}_e(n-1)$.

Thus, it remains to consider the overpartitions in $\overline{\mathrm{Spt}}1_o(n)\bigcup \overline{\mathrm{Spt}}1_o(n-2)$. We construct the following map from this set of overpartitions to $\overline{P}_e(n-1)\bigcup \overline{P}_{oex}(n-1)$.

If $\pi\in \overline{\mathrm{Spt}}1_o(n)$ with $s(\pi)>1$, then
$$\pi\to
    \begin{cases}
        \left(\ldots,\overline{s(\pi)-1}\right), & \text{ if } s(\pi) \text{ even,}\\
        &\hspace{-15mm} (\textit{which belongs to } \overline{P}_{oex}(n-1))\\
        \left(\ldots,\overline{s(\pi)-1}\right), & \text{ if } s(\pi) \text{ odd,}\\
        &\hspace{-15mm} (\textit{which belongs to } \overline{P}_e(n-1))
    \end{cases}
$$
and if $\pi\in \overline{\mathrm{Spt}}1_o(n-2)$, then
$$\pi\to
    \begin{cases}
        \left(\ldots,s(\pi)+1\right), & \text{ if } s(\pi) \text{ even,}\\
           & \hspace{-15mm} (\textit{which belongs to } \overline{P}_{oex}(n-1))\\
        \left(\ldots,s(\pi)+1\right), & \text{ if } s(\pi) \text{ odd.}\\
          &\hspace{-15mm} (\textit{which belongs to } \overline{P}_e(n-1))
    \end{cases}
$$
This completes the proof.

An example is shown in the following figure.

\noindent	\begin{tikzpicture}[
		box/.style={draw, rounded corners=2pt, inner sep=5pt, font=\small, align=center},
		container/.style={draw, thick, rounded corners=5pt, inner sep=8pt},
		arrow/.style={-Stealth, thick, shorten >=2pt, shorten <=2pt}
		]
		\node[box] (spt7_top) {$(6,1)$, \quad $(\overline{6},1)$ \\
			 $(4,2,1)$, \quad $(\overline{4},2,1)$\\
			 $(4,\overline{2},1)$, \quad $(\overline{4},\overline{2},1)$\\
			 $(2,2,2,1)$, \quad $(\overline{2},2,2,1)$};
		\node[box, below left=0.5cm and -0.8cm of spt7_top] (spt7_bl) {$(5,2)$ \\ $(\overline{5},2)$};
		\node[box, below right=0.5cm and -1.8cm of spt7_top] (spt7_br) {$(7)$ \\
			$(4,3)$, \quad $(\overline{4},3)$};
		\node[container, fit=(spt7_top) (spt7_bl) (spt7_br), label=below:{$\hspace{-6mm}\overline{\mathrm{Spt}}1_o(7)$}] (Spt7) {};

		\node[box, below=1.4cm of Spt7] (spt5_top) {$(3,2)$\\ $(\overline{3},2)$};
		\node[box, below=0.3cm of spt5_top] (spt5_bot) {$(5)$\\
			$(4,1)$, \quad $(\overline{4},1)$\\
			$(2,2,1)$, \quad $(\overline{2},2,1)$};
		\node[container, fit=(spt5_top) (spt5_bot), label=below:{$\overline{\mathrm{Spt}}1_o(5)$}] (Spt5) {};
		
		\node[box, right=3cm of spt7_top, yshift=-0.7cm] (pex6_top) {$(6)$, \quad $(4,2)$, \quad $(\overline{4},2)$\\
			$(2,2,2)$, \quad $(\overline{2},2,2)$};
		\node[box, below=0.3cm of pex6_top] (pex6_bot) {$(\overline{6})$\\
			$(4,\overline{2})$, \quad $(\overline{4},\overline{2})$};
		\node[container, fit=(pex6_top) (pex6_bot), label=below:{$\overline{P}_e(6)$}] (Pex6) {};
		
		\node[box, right=5.0cm of spt5_top, yshift=-0.2cm] (poex6_top) {$(3,3)$, \quad $(\overline{3},3)$};
		\node[box, below=0.3cm of poex6_top] (poex6_bot) {$(5,\overline{1})$, \quad $(\overline{5},\overline{1})$};
		\node[container, fit=(poex6_top) (poex6_bot), label=below:{$\overline{P}_{oex}(6)$}] (Poex6) {};
		
		\draw[arrow] (spt7_top.east) to [out=35, in=105] (Pex6.north);
		\draw[arrow] (spt7_bl.east) -- (poex6_bot.west);
		\draw[arrow] (spt7_br.east) -- (pex6_bot.west);
		
		\draw[arrow] (spt5_top.east) -- (poex6_top.west);
		\draw[arrow] (spt5_bot.east) -- (pex6_top.west);
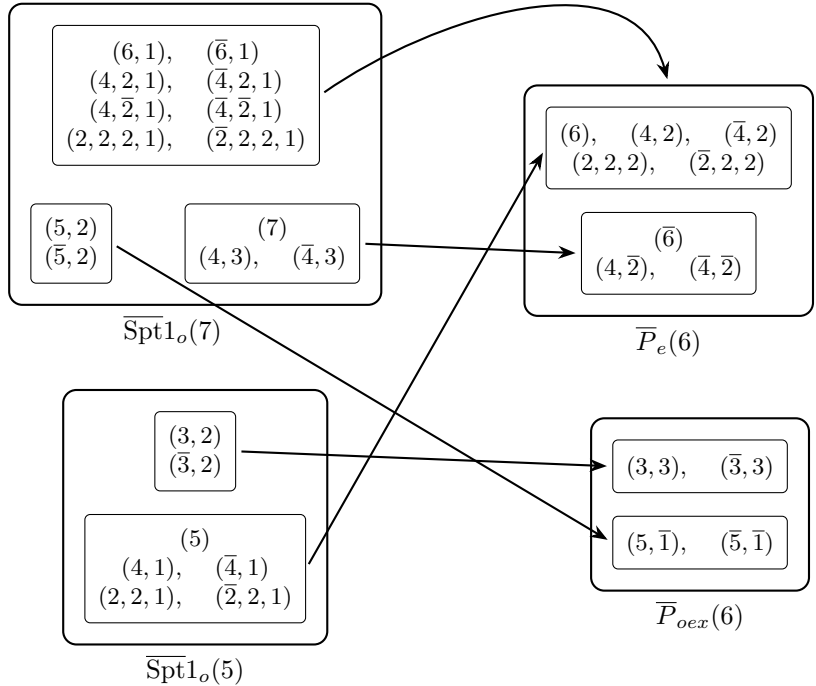
\node at (4.0, -8.1) {Example for $n= 7$.};		
	\end{tikzpicture}

\end{proof}

\begin{proof}[Proof of Theorem \ref{cor4.3}]
Here, we will use the definitions of SAME nature and OPPOSITE nature as we introduced in the article \cite{ndb2026pjm}.

Let $\pi\in \overline{\mathrm{Spt}}1_o(n)$. We construct the following map.
$$\pi\to
\begin{cases}
    \left(\ldots,s_2(\pi)-1\right), & \text{ if } s(\pi)=1 \text{ and } s_2(\pi) \text{ is not overlined},\\
    &\hspace{-20mm} (\textit{This overpartition belongs to } \overline{\mathrm{Spt}}1_o(n-2) \textit{ and is of OPPOSITE nature.})\\
    \left(\ldots,s_2(\pi)+1\right), & \text{ if } s(\pi)=1 \text{ and } s_2(\pi) \text{ is overlined}.\\
    &\hspace{-20mm} (\textit{This overpartition belongs to } \overline{\mathrm{Spt}}1_o(n) \textit{ and is of OPPOSITE nature.})
\end{cases}
$$
Thus, the overpartitions with even $s(\pi)$ remain in $\overline{\mathrm{Spt}}1_o(n)\bigcup\overline{\mathrm{Spt}}1_o(n-2)$. We now construct a map from this subset to $\overline{P}_{oex}(n-1)$.
$$\pi\to
\begin{cases}
    \left(\ldots,\overline{s(\pi)-1}\right), & \text{ if } \pi\in \overline{\mathrm{Spt}}1_o(n),\\
    &\hspace{-15mm} (\textit{OPPOSITE nature.})\\
    \left(\ldots,s(\pi)+1\right), & \text{ if } \pi\in \overline{\mathrm{Spt}}1_o(n-2).\\
    &\hspace{-15mm} (\textit{OPPOSITE nature.})
\end{cases}
$$
This completes the proof.

An example is shown in the following figure.

\noindent\begin{tikzpicture}[box/.style={draw, rounded corners, inner sep=5pt, font=\small},
		group/.style={draw, thick, rounded corners, inner sep=10pt},
		arrow/.style={-Stealth, thick, shorten >=2pt, shorten <=2pt}]
		
		\node[box] (L1) at (0,4) {\shortstack{$(6,2,1)$ \\ $(\overline{6},2,1)$ \\ $(4,4,1)$ \\ $(\overline{4},4,1)$ \\ $(2,2,2,2,1)$ \\ $(\overline{2},2,2,2,1)$}};
		\node[box, right=0.5cm of L1] (L2) {\shortstack{$(8,1)$ \\ $(4,2,2,1)$ \\ $(\overline{4},2,2,1)$ \\ $(4,\overline{2},2,1)$ \\ $(\overline{4},\overline{2},2,1)$}};
		\node[box, below=0.5cm of L1] (L3) {\shortstack{$(6,\overline{2},1)$ \\ $(\overline{6},\overline{2},1)$}};
		\node[box, below=0.3cm of L2] (L4) {$(\overline{8},1)$};
		\node[box, below=0.2cm of L4] (L5) {\shortstack{$(6,3)$ \\ $(\overline{6},3)$}};
		\node[box, below=0.8cm of L3] (L6) {$(9)$};
		\node[box, below=0.1cm of L5] (L7) {\shortstack{$(7,2)$ \\ $(\overline{7},2)$ \\ $(5,4)$ \\ $(\overline{5},4)$}};
		
		\node[group, fit=(L1) (L2) (L6) (L7)] (G9) {};
		\node[below=2pt of G9] {$\overline{\mathrm{Spt}}1_o(9)$};
		
		\node[box] (R1) at (7,4) {\shortstack{$(7)$ \\ $(4,2,1)$ \\ $(\overline{4},2,1)$ \\ $(4,\overline{2},1)$ \\ $(\overline{4},\overline{2},1)$}};
		\node[box, right=0.5cm of R1] (R2) {\shortstack{$(6,1)$ \\ $(\overline{6},1)$ \\ $(4,3)$ \\ $(\overline{4},3)$ \\ $(2,2,2,1)$ \\ $(\overline{2},2,2,1)$}};
		\node[box, below=0.5cm of R2] (R3) {\shortstack{$(5,2)$ \\ $(\overline{5},2)$}};
		
		\node[group, fit=(R1) (R2) (R3)] (G7) {};
		\node[below=2pt of G7] {$\overline{\mathrm{Spt}}1_o(7)$};
		
		\node[box] (B1) at (7.4,-1.7) {\shortstack{$(7,\overline{1})$ \\ $(\overline{7},\overline{1})$ \\ $(5,\overline{3})$ \\ $(\overline{5},\overline{3})$}};
		\node[box, right=0.8cm of B1] (B2) {\shortstack{$(5,3)$ \\ $(\overline{5},3)$}};
		
		\node[group, fit=(B1) (B2)] (G8) {};
		\node[below=2pt of G8] {$\overline{P}_{oex}(8)$};
		
		\draw[arrow] (L1.north) to[bend left=20] (R2.north);
		\draw[arrow] (L2.north) to[bend left=20] (R1.north);
		
		\draw[arrow] (L3.east) -- (L5.west);
		\draw[arrow] (L4.south west) -- (L6.north east);
		
		\draw[arrow] (L7.east) -- (B1.west);
		\draw[arrow] (R3.south) -- (B2.north);
		\node at (4.5, -3.9) {Example for $n= 9$.};
	\end{tikzpicture}

\end{proof}

\section{Combinatorial proof of Theorem \ref{thm4}}\label{sec:proof4}

\begin{proof}[Combinatorial proof of \eqref{id:e}]
We construct a bijection from $B_e(1,n)\bigcup B_e(1,n-2)$ to $\overline{P}_e(n-1)\bigcup C_o(n-1)$ as follows.

For given $\pi\in B_e(1,n)\bigcup B_e(1,n-2)$, there are two cases.
\begin{description}
\item[Case I] $s(\pi)\equiv 0 \pmod 2$. Then,
$$\pi\to
    \begin{cases}
        \left(\ldots,\overline{s(\pi)-1}\right), & \text{ if } \pi\in B_e(1,n),\\
        \left(\ldots,s(\pi)+1\right), & \text{ if } \pi\in B_e(1,n-2).
    \end{cases}
$$

\item[Case II] $s(\pi)\equiv 1 \pmod 2$. Then,
$$\pi\to
    \begin{cases}
        \begin{cases}
        \left(\ldots,s_2(\pi)\right),  & \text{ when } s(\pi)=1\\
            \left(\ldots,\overline{s(\pi)-1}\right), & \text{ when } s(\pi)>1
        \end{cases} \ \ ,
        & \text{ if } \pi\in B_e(1,n),\\
        \left(\ldots,s(\pi)+1\right), & \text{ if } \pi\in B_e(1,n-2).
    \end{cases}
$$
\end{description}

All image overpartitions $\mu$ are distinct in both cases. Moreover, $\mu\in C_o(n-1)$ in Case I and $\mu\in \overline{P}_e(n-1)$ in Case II. This completes the proof.

An example is shown in the following figure.\\

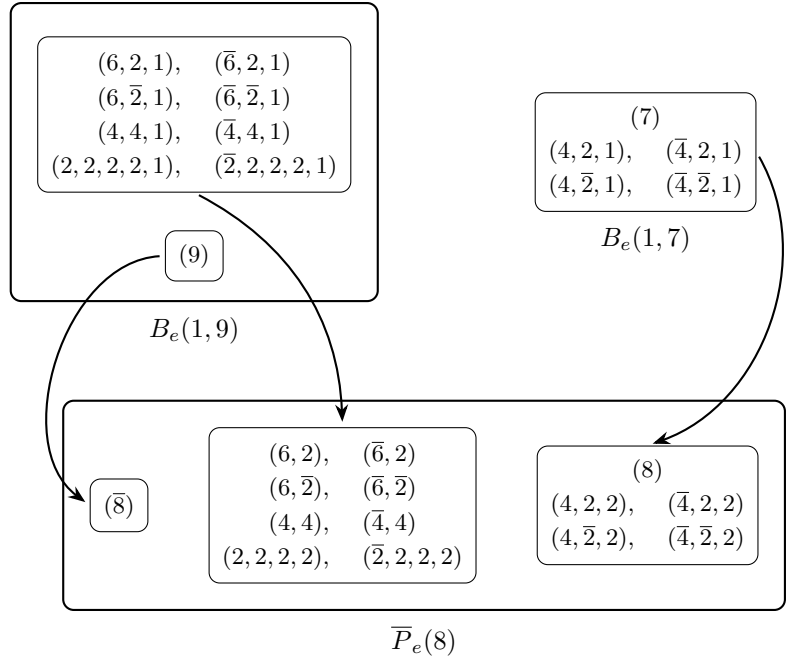
\noindent\begin{tikzpicture}[box/.style={draw, rounded corners, inner sep=5pt, font=\small},
	group/.style={draw, thick, rounded corners, inner sep=10pt},
	arrow/.style={-Stealth, thick, shorten >=2pt, shorten <=2pt}]
	
	\node[box] (L1)  at (2,3.5) {\shortstack{$(6,2,1)$, \quad $(\overline{6},2,1)$\\
			$(6,\overline{2},1)$, \quad  $(\overline{6},\overline{2},1)$ \\
			$(4,4,1)$, \quad $(\overline{4},4,1)$ \\
			$(2,2,2,2,1)$, \quad $(\overline{2},2,2,2,1)$}};
	\node[box, below=0.5cm of L1] (L2) {$(9)$};
	\node[group, fit=(L1) (L2)] (G9) at (2,3) {};
	\node[below=2pt of G9] {$B_e(1,9)$};
	
	\node[box] (G7) at (8,3) {\shortstack{$(7)$ \\
			$(4,2,1)$, \quad $(\overline{4},2,1)$ \\
			 $(4,\overline{2},1)$, \quad $(\overline{4},\overline{2},1)$}};
	\node[below=2pt of G7] {$B_e(1,7)$};
	  
	\node[box] (B1) at (1.0,-1.7) {\shortstack{$(\overline{8})$}};
	\node[box, right=0.8cm of B1] (B2) {\shortstack{$(6,2)$, \quad $(\overline{6},2)$\\
			$(6,\overline{2})$, \quad  $(\overline{6},\overline{2})$ \\
			$(4,4)$, \quad $(\overline{4},4)$ \\
			$(2,2,2,2)$, \quad $(\overline{2},2,2,2)$}};
		\node[box, right=0.8cm of B2] (B3) {\shortstack{$(8)$ \\
				$(4,2,2)$, \quad $(\overline{4},2,2)$ \\
				$(4,\overline{2},2)$, \quad $(\overline{4},\overline{2},2)$}};
	
	\node[group, fit=(B1) (B2) (B3)] (G8) {};
	\node[below=2pt of G8] {$\overline{P}_e(8)$};

	\draw[arrow] (L1.south) to[bend left=30] (B2.north);
	\draw[arrow] (L2.west) to[bend right=70] (B1.west);
	\draw[arrow] (G7.east) to[bend left=50] (B3.north);
	\node at (5.5, -4.3) {Example for $n= 9$.};
\end{tikzpicture}

\end{proof}

The proof of \eqref{id:o} is similar to that of \eqref{id:e}. So, we omit. Theorem \ref{cor4.2} can be derived from Theorem \ref{thm4} as follows. We have,
\begin{align*}
\overline{\mathrm{spt}}1_o(n)+\overline{\mathrm{spt}}1_o(n-2)
&=(b_e(1,n)+b_o(1,n))+(b_e(1,n-2)+b_o(1,n-2))\\
&=(b_e(1,n)+b_e(1,n-2))+(b_o(1,n)+b_o(1,n-2))\\
&=2\overline{p}_e(n-1)+c_o(n-1)+c_e(n-1) \text{\quad (using \eqref{id:e} and \eqref{id:o})}\\
&=2\overline{p}_e(n-1)+\overline{p}_{oex}(n-1).
\end{align*}
Theorem \ref{cor4.3} follows in a similar way.


\section*{Acknowledgements}
The third author was partially supported by an institutional fellowship for doctoral research from Tezpur University, Assam, India.

\end{document}